\documentclass{amsart}
\usepackage[utf8]{inputenc}

\usepackage{amsmath,amsfonts, amssymb}
\usepackage{graphicx}
\usepackage{xcolor}
\usepackage{enumerate}
\usepackage{enumitem}
\usepackage{cite}
\usepackage{tabu}

\usepackage{stmaryrd} 

\usepackage{fullpage}
\usepackage[active]{srcltx}

\usepackage{changepage}
\usepackage{chngcntr}

\usepackage{multibib}

\usepackage{url}
\usepackage{hyperref}
\hypersetup{
	colorlinks,
	linkcolor={red!50!black},
	citecolor={blue!50!black},
	urlcolor={blue!80!black}
}  

\theoremstyle{plain}
\newtheorem{theorem}{Theorem}[section]

\newtheorem{cor}[theorem]{Corollary}
\newtheorem{conj}[theorem]{Conjecture}
\newtheorem{prop}[theorem]{Proposition}
\newtheorem{lemma}[theorem]{Lemma}

\newtheorem{question}[theorem]{Question}

\newtheorem*{maintheorem}{Theorem~\ref{thm:main}}
\newtheorem*{ribbonfusion}{Theorem~\ref{thm:murasugisums}}
\newtheorem*{subgraphTRP}{Theorem~\ref{thm:subgraphTRP}}

\theoremstyle{definition}
\newtheorem{defn}[theorem]{Definition}
\newtheorem{remark}[theorem]{Remark}

\newcommand{\calF}{\ensuremath{{\mathcal F}}}

\newcommand{\comment}[1]{}

\newcommand{\bdry}{\ensuremath{\partial}}

\newcommand{\R}{\ensuremath{\mathbb{R}}}
\newcommand{\Z}{\ensuremath{\mathbb{Z}}}

\newcommand{\cut}{\ensuremath{\backslash}}

\newcommand{\tb}{\ensuremath{{\mbox{\tt tb}}}}
\newcommand{\rot}{\ensuremath{{\mbox{\tt rot}}}}

\definecolor{amaranth}{rgb}{0.9, 0.17, 0.31} 
\definecolor{carrotorange}{rgb}{0.93, 0.57, 0.13} 
\definecolor{citrine}{rgb}{0.89, 0.82, 0.04} 
\definecolor{dartmouthgreen}{rgb}{0.05, 0.5, 0.06} 
\definecolor{ballblue}{rgb}{0.13, 0.67, 0.8} 
\definecolor{ceruleanblue}{rgb}{0.16, 0.32, 0.75} 
\definecolor{amethyst}{rgb}{0.6, 0.4, 0.8} 
\definecolor{amber}{rgb}{1.0, 0.75, 0.0} 
\definecolor{burlywood}{rgb}{0.87, 0.72, 0.53} 

\title{Non-looseness of boundaries of Legendrian ribbons}

\makeatletter
\@namedef{subjclassname@1991}{2020 Mathematics Subject Classification}
\makeatother

\author{Kenneth L. Baker}
\address{Department of Mathematics, 
University of Miami}   
\email{k.baker@math.miami.edu}  

\author{Sinem Onaran}
\address{Department of Mathematics,
Hacettepe University}   
\email{sonaran@hacettepe.edu.tr} 

\begin{document}

\begin{abstract}
Every null-homologous link in an oriented $3$-manifold is isotopic to the boundary of a ribbon of a Legendrian graph for any overtwisted contact structure.  However this is not the case if the boundary is required to be non-loose.   
Here, we define the `Tight Reattachment Property' for a Legendrian graph and show that it implies the boundary of its ribbon is non-loose.  We also discuss the applicability of this property and examine examples and constructions of Legendrian graphs with this property. 
\end{abstract}

\maketitle

\section{Introduction}

The notion that a link in $S^3$ is strongly quasipositive, SQP, has been extended to other $3$-manifolds via open books by several sets of authors in essentially the same way \cite{BEHKVHM, ItoKawamuro, Hayden-QPlinksandSteinsfces}. In turn, such links correspond to boundaries of ribbons of Legendrian graphs in the contact structure supported by the open book.  

In a contact $3$-manifold $(M,\xi)$, a graph $\Lambda$ is {\em Legendrian} if it is everywhere tangent to $\xi$.  Its {\em (Legendrian) ribbon}  $R(\Lambda)$ is an embedded compact surface that is tangent to $\xi$ along $\Lambda$, is otherwise transverse to $\xi$, and retracts to $\Lambda$ under a flow tangent to $\xi\vert_{R(\Lambda)}$.  As such $\bdry R(\Lambda)$, the boundary of the ribbon, is a transverse link.

Baader-Ishikawa show that quasipositive surfaces in $S^3$ correspond to ribbons of connected Legendrian graphs in the standard tight contact structure $(S^3, \xi_{std})$ \cite{BaaderIshikawa}. 
Consequently, strongly quasipositive links in $S^3$ are exactly those links that arise as the boundary of the ribbon of a connected Legendrian graph in $(S^3, \xi_{std})$.  
(See \cite{rudolph} for an overview of quasipositive surfaces and strongly quasipositive links.)

Generalizing this to other closed manifolds, Hayden showed that a link is SQP with respect to an open book if and only if it is the boundary of a ribbon of a Legendrian graph in the contact structure supported by the open book \cite[Theorem 1]{Hayden-LegribbonsandSQP}.  
However, while the property of being SQP  places significant restrictions on links in tight contact structures, every null-homologous link is the boundary of a Legendrian ribbon in an overtwisted contact structure on a closed $3$-manifold \cite{BCV-LegRibbonsinOTctctstr}.
Thus this notion of SQP for overtwisted contact structures is no longer discriminating the way it is for tight contact structures.

What enables this failure of discrimination is that the complement of the Legendrian graph is permitted to be overtwisted.  
To regain some of the strength of the theory of strongly quasipositive links in tight manifolds, we shift focus to  Legendrian graphs and their ribbons with tight complement.

\medskip
The term `non-loose' is typically reserved for subsets of overtwisted contact manifolds with tight complement, but we find it convenient to relax this constraint.
Throughout this article, let us say a subset of a contact $3$-manifold is {\em non-loose} if its complement is tight, even if the ambient contact structure is not overtwisted.  
Also, following \cite{stipsicz-vertesi}, we say a subset is {\em strongly non-loose} if its complement is tight and has no Giroux torsion.
With this terminology, subsets of tight manifolds are necessarily strongly non-loose.

As a starting point, 
one readily observes that a Legendrian graph 
$\Lambda$ and its ribbon $R(\Lambda)$ have regular neighborhoods with contactomorphic exteriors.  Consequently, if $\Lambda$ is loose if and only if $R(\Lambda)$ is too, and these further imply that $\bdry R(\Lambda)$ is loose as well.  Therefore we have:
\begin{lemma}\label{lem:loosenessofribbonsandspines}
For a Legendrian graph $\Lambda$,\\
\phantom{xxxxxxxx}
$\bdry R(\Lambda)$ is non-loose $\implies$ $\Lambda$ is non-loose $\iff$ $R(\Lambda)$ is non-loose.
\qed
\end{lemma}
However, the first implication doesn't immediately go the other way.
To what extent do we have the missing implication?
\begin{question}\label{ques:nonlooseboundaryofribbon}
If $\Lambda$ is a non-loose Legendrian graph, is the transverse link $\bdry R(\Lambda)$ also non-loose? 
\end{question}

As evidence for the possibility that there may exist an example of a non-loose Legendrian graph $\Lambda$ with $\bdry R(\Lambda)$ loose, we observe that there are non-loose Legendrian {\em knots} $\Lambda$ with positive transverse push-offs $T_+(\Lambda)$ that are loose.  However in this situation $\bdry R(\Lambda)$ is the double transverse push-off $T_+(\Lambda) \cup T_-(\Lambda)$ which is non-loose in all the cases we know.  Basic examples of this behavior come from the non-loose Legendrian unknots in an overtwisted $S^3$, see \cite{EliFr} (also presented in Theorem~\ref{thm:unknotclassification}) and the discussion in \cite{Etnyre-knotsinOTctctstr}. Similar examples can be found for other connected Legendrian graphs $\Lambda$ in which $\bdry R(\Lambda)$ is a transverse link of more than one component and some proper sublink is loose.

If the Legendrian graph $\Lambda$ is more than just non-loose but actually has universally tight exterior, then there is an affirmative answer to Question~\ref{ques:nonlooseboundaryofribbon}.  It follows from Convex Decomposition Theory and Theorem 2.7 of \cite{HKM-convexdecomptheory} that the Legendrian approximations of $\bdry R(\Lambda)$ all have universally tight exterior.  Hence $\bdry R(\Lambda)$ is non-loose.

Hedden and Tovstopyat-Nelip  give another sort of answer, showing that every null-homologous link in a closed oriented $3$-manifold is the non-loose boundary of a ribbon for some contact structure. 
\begin{theorem}[\cite{HTN}]
For any null-homologous link $L$ in a closed oriented $3$-manifold $M$, there is some contact structure $\xi$ on $M$ and Legendrian graph $\Lambda$ so that $\bdry R(\Lambda)$ is non-loose and isotopic to $L$.
\end{theorem}

\begin{proof}[Sketch of proof]
Given a minimal genus Seifert surface $F$ for $L$, the complementary sutured manifold $(M_F, \gamma_F)$ is taut and therefore admits a sutured manifold hierarchy \cite{gabai}. This hierarchy can be turned into a convex decomposition hierarchy \cite{HKM-convexdecomptheory} in which the convex decomposing surfaces all have non-nested boundary parallel dividing curves. Then tightness of the resulting product pieces can be pulled back to a tight contact structure on $M_F$ with convex boundary corresponding to the sutured structure.  Finally the positive and negative regions of the boundary, $R_+(\gamma_F)$ and $R_-(\gamma_F)$, may be viewed as ribbons of Legendrian graphs $\Lambda_+$ and $\Lambda_-$ that can be reglued to a Legendrian graph $\Lambda$ whose ribbon $R(\Lambda)$ is isotopic to $F$.  The tightness also pulls back through this gluing giving a tight contact structure in the complement of $\bdry R(\Lambda)$.  As $F$ is isotopic to $R(\Lambda)$, $L$ is isotopic to $\bdry R(\Lambda)$.
\end{proof}

There are potentially many choices in the above sketch, leading to the link $L$ being isotopic to a non-loose link $\bdry R(\Lambda)$  for some Legendrian graph $\Lambda$ in many different contact structures $\xi$.   
However, not every Legendrian ribbon with tight complement arises from this construction. Hedden and Tovstopyat-Nelip further show that the transverse link $\bdry R(\Lambda)$ constructed in their proof always has non-zero $\hat{t}$ invariant \cite{HTN}, extending a previous result for the transverse bindings of open books \cite{TN2020Transverse}. 
See \cite{BVVV} for the definition of the transverse invariant $\hat{t}$. 

\begin{remark}\label{rem:trivialthat}
Non-loose transverse links $\bdry R(\Lambda)$ with trivial $\hat{t}$ invariant do exist. Examples may be constructed by taking connected sums with tight contact manifolds that have trivial contact invariant.

More specifically,
Let $\Lambda_1$ be a Legendrian graph in $(M_1,\xi_1)$ with $\bdry R(\Lambda_1)$ non-loose.  Let $(M_0,\xi_0)$ be a tight contact manifold with trivial contact invariant.  Then, letting $\Lambda$ be the image of $\Lambda_1$ in the connected sum $(M_1\#M_0,\xi_1\#\xi_0)$, 
the transverse link $\bdry R(\Lambda_1)$ is non-loose since its exterior is the connected sum of tight manifolds (eg.\ Lemma~\ref{lem:connectedsum}).  However $\hat{t}(\bdry R(\Lambda_1))=0$ due to the behavior of these contact invariants under connected sum.
\end{remark}

\medskip

\subsection{The Tight Reattachment Property}
Our main theorem gives an affirmative answer to Question~\ref{ques:nonlooseboundaryofribbon}
in a different setting, when the Legendrian graph satisfies a condition that we call ``the Tight Reattachment Property''.  To state it, we need a few definitions.

In a contact $3$-manifold $(M,\xi)$, let $\Lambda$ be a Legendrian graph with regular closed neighborhood $N(\Lambda)$. That is, $N(\Lambda)$ is a thickening of a ribbon $R(\Lambda)$ to a handlebody with convex boundary whose dividing curves are $\bdry R(\Lambda)$.  Then, letting $\xi_\Lambda$ be the restriction of $\xi$ to $M_\Lambda = M\cut int N(\Lambda)$, the contact manifold $(M_\Lambda,\xi_\Lambda)$ is the {\em exterior} of $\Lambda$ and has convex boundary $\bdry N(\Lambda)$.

We may reglue $\bdry_- N(\Lambda)$ to $\bdry_+ N(\Lambda)$ by any orientation preserving diffeomorphism $\psi$ that is the identity in a collar of the dividing curves $\bdry R(\Lambda)$ to make a new contact $3$-manifold $(M_{\Lambda,\psi},\xi_{\Lambda,\psi})$.   Say $(M_{\Lambda,\psi},\xi_{\Lambda,\psi})$ is obtained from $(M,\xi)$ by a {\em reattachment along $\Lambda$}. If some reattachment along $\Lambda$ is tight, then we say $\Lambda$ has the {\em Tight Reattachment Property}, TRP for short.  Observe that a Legendrian graph with the TRP necessarily has tight complement as its exterior embeds in a tight manifold.    

As these reattachments are really about the ribbon, we may also speak of a ribbon having the TRP.  Indeed, the ribbon $R(\Lambda)$ has the same exterior as the Legendrian graph $\Lambda$.   To that point, Lemma~\ref{lem:legspinesofribbons} shows that any spine of a ribbon can be Legendrian realized to have an isotopic ribbon.
Furthermore, one may view the exterior $(M_\Lambda,\xi_\Lambda)$ of $\Lambda$ as being supported by a partial open book \cite{HKM-ctctinvtSFH}. In the language of \cite{licatamathews}, the TRP asks that such a partial open book extends to an open book supporting a tight contact structure.



With the definition of TRP, we may now state our main theorem.
\begin{maintheorem}
    Let $\Lambda$ be a connected Legendrian graph in a closed contact $3$-manifold.  If $\Lambda$ has the Tight Reattachment Property,
    then the transverse link $\bdry R(\Lambda)$ is non-loose.
\end{maintheorem}
This theorem developed from and generalizes an initial approach for determining non-looseness of boundaries of Legendrian ribbons using Heegaard Floer contact invariants which we present is Section~\ref{sec:LOSS}.

To show this property applies somewhat broadly, give two constructions of Legendrian graphs with the TRP.   
Observe that one may express Murasugi sums of ribbons as ribbons of `Legendrian fusions' of Legendrian graphs, see section~\ref{sec:fusion}.  
\begin{ribbonfusion}
    Suppose $\Lambda$ is the Legendrian fusion $\Lambda_+ \ast \Lambda_-$ of two Legendrian graphs.
    If each $\Lambda_+$ and $\Lambda_-$ have the Tight Reattachment Property, then $\Lambda$ has the Tight Reattachment Property.
\end{ribbonfusion}
Restated for ribbons, together these give the following:
\begin{cor}
    Suppose $R_i$ is a connected Legendrian ribbon  with the Tight Reattachment Property in $(M_i,\xi_i)$ for each $i=1,2$.  Then any Murasugi sum $R_1 \ast R_2$ may be regarded as a connected Legendrian ribbon $R$ with the Tight Reattachment Property so that 
    the transverse link $\bdry R$ is non-loose. \qed
\end{cor}

We further observe that the TRP is inherited from subgraphs.
\begin{subgraphTRP}
    Let $\Lambda'$ be a connected subgraph of a connected Legendrian graph $\Lambda$ in a closed contact $3$-manifold.  If $\Lambda'$ has the TRP, then $\Lambda$ does too.
\end{subgraphTRP}
In general there is no reason that the non-looseness of $\bdry R(\Lambda')$ should confer the non-looseness of $\bdry R(\Lambda)$.
However, with Theorem~\ref{thm:main}, Theorem~\ref{thm:subgraphTRP} shows the non-looseness of $\bdry R(\Lambda)$ follows from this stronger sense of  non-looseness of $\bdry R(\Lambda')$.

\medskip

If a Legendrian graph $\Lambda$ is actually a Legendrian knot, then $\Lambda$ has the TRP if and only if some sequence of Legendrian surgeries (either positive or negative) on $\Lambda$ yields a tight manifold.  Thus we immediately have the following corollary.

\begin{cor}
    Let $L$ be a Legendrian knot in a closed contact $3$-manifold.  If some sequence of $(\pm 1)$ Legendrian surgeries on $L$ yields a tight contact $3$-manifold, then the transverse link $\bdry R(L)$ is non-loose. \qed
\end{cor}

\begin{remark}\label{rem:massotex}
    Not every non-loose Legendrian knot has the TRP.  
    Since any contact surgery on a non-loose Legendrian knot with boundary parallel full Giroux torsion will be overtwisted, such knots cannot have the TRP. 
    Furthermore there are even examples of {\em strongly} non-loose Legendrian knots in closed contact manifolds for which any sequence of $(\pm1)$ Legendrian surgeries is overtwisted; therefore such knots do not have the TRP.  
    
    One may observe that the Legendrian knots of \cite[Theorem 1]{massot} actually are such examples of strongly non-loose Legendrian knots without the TRP.  Moreover, for one orientation, the transverse pushoff is non-loose  (but with boundary parallel full Giroux torsion).  Consequently, the boundary of a ribbon of such Legendrian knots is non-loose, but that fact doesn't follow from our Theorem~\ref{thm:main}.

    More generally, one should be able to create such examples by, say, starting with a transverse knot in a tight contact structure, inserting full Giroux torsion along it, and then taking a parallel Legendrian curve on the Giroux torsion torus out where the planes have rotated by $\pi/2$.  Such a Legendrian knot cannot have the TRP and we expect it to also be strongly non-loose.  
\end{remark}

\begin{question}
    Let $\Lambda$ be a non-loose connected Legendrian graph.  If $\bdry R(\Lambda)$ is non-loose but $\Lambda$ does not have the TRP, then is $\Lambda$ a knot?
\end{question}

\begin{question}
        A Legendrian graph with the TRP is necessarily non-loose.  Since some sequence of stabilizations will loosen a Legendrian graph in an overtwisted contact structure (cf.\ \cite{BO-nonlooseness}), the TRP cannot be preserved by stabilization.  
        However, is the TRP preserved by Legendrian destabilization?
\end{question}

\subsection{The Bennequin bound}
The Bennequin bound was originally stated by Bennequin for transverse knots in $(S^3, \xi_{std})$ \cite{bennequin} and extended by Eliashberg to null-homologous transverse links in any tight $3$-manifold \cite{eliashberg}.  Its extension to non-loose null-homologous transverse links in any contact $3$-manifold is attributed to \'Swiatowski; see \cite[Proposition 1.3]{Etnyre-knotsinOTctctstr} and \cite[Theorem 1.5]{ItoKawamuro}.  
\begin{prop}[The Bennequin bound]\label{prop:bennybound}
   Let $T$ be a non-loose, null-homologous transverse link.  Then for any Seifert surface $\Sigma$ we have  $sl(T, [\Sigma]) \leq -\chi(\Sigma)$. \qed
\end{prop}

As observed in \cite[Lemma 2.2]{BCV-LegRibbonsinOTctctstr}, one readily sees that
for a Legendrian graph $\Lambda$ we have 
\[sl(\bdry R(\Lambda), [R(\Lambda)]) = -\chi(R(\Lambda)).\]
As such, for null-homologous transverse links in tight manifolds, it has been conjectured by many that the Bennequin bound is always realized by a such a ribbon, see for example {\cite[Conjecture 4.1]{Hayden-LegribbonsandSQP}} and the conjectures in \cite{ItoKawamuro}.

\begin{conj}
    Suppose a transverse null-homologous link $T$ in a tight $3$-manifold $(M,\xi)$ realizes the Bennequin bound.  
    Then there is a Legendrian graph $\Lambda$ so that $T = \bdry R(\Lambda)$.  
\end{conj}

In general, the presence of Giroux torsion in a link exterior can prevent this from holding true for non-loose links.  However, perhaps the conjecture continues to hold for strongly non-loose links.

\begin{conj}
    Suppose a transverse null-homologous link $T$ in a contact $3$-manifold $(M,\xi)$ is strongly non-loose and realizes the Bennequin bound.  Then there is a Legendrian graph $\Lambda$ so that $T = \bdry R(\Lambda)$.
\end{conj}

\begin{lemma}
    If the transverse link $\bdry R(\Lambda)$ is non-loose then $R(\Lambda)$ minimizes genus among Seifert surfaces for $\bdry R(\Lambda)$ in the same relative homology class.  
\end{lemma}

\begin{proof}
    Since $sl(\bdry R(\Lambda), [R(\Lambda)]) = -\chi(R(\Lambda))$ as observed above, 
    any other Seifert surface $\Sigma$ for $L$ that is homologous to $R(\Lambda)$ would have to satisfy $sl(\bdry R(\Lambda), [R(\Lambda)]) \leq -\chi(\Sigma)$ by the Bennequin bound of Proposition~\ref{prop:bennybound}.  Hence $R(\Lambda)$ minimizes genus.
\end{proof}

\begin{question}
    If a Legendrian graph $\Lambda$ is non-loose but $\bdry R(\Lambda)$ is loose, must $R(\Lambda)$ still minimize genus in its homology class? 
\end{question}

\begin{remark}
    Note that if $R(\Lambda)$ were compressible, then one may find an overtwisted disk in a convex product neighborhood over $R(\Lambda)$.
    Furthermore, if $\bdry R(\Lambda)$ is loose while $R(\Lambda)$ is not minimal genus, then the proof of Proposition~\ref{prop:bennybound} (see eg.\ proof of \cite[Theorem 4.15]{Etnyre-IntroContactLect}) adapts to show that any minimal genus Seifert surface for $\bdry R(\Lambda)$ is loose.
\end{remark}

\section{Legendrian graphs and non-looseness}

For the fundamentals of contact topology including the theory of Legendrian and transverse links as well as convex surface theory, we refer the reader to  \cite{GeiBook}.
For the basics of the theory of non-loose knots in overtwisted contact manifolds, we refer the reader to \cite{Etnyre-knotsinOTctctstr}.

Let us also recall the following elementary lemma.

\begin{lemma}\label{lem:connectedsum}
    Let $(M_1, \xi_1)$ and $(M_2,\xi_2)$ be contact $3$--manifolds.  If the contact connected sum $(M_1 \# M_2, \xi_1 \# \xi_2)$ is tight, then both $(M_1, \xi_1)$ and $(M_2,\xi_2)$ are tight.
\end{lemma}

\begin{proof}
    Say $(M_1, \xi_1)$ is not tight.  Then it contains an overtwisted disk. Moreover, it contains an overtwisted disk in the complement of any Darboux ball.   Thus the overtwisted disk persists in the connected sum with $(M_2,\xi_2)$ making $(M_1 \# M_2, \xi_1 \# \xi_2)$ overtwisted.
\end{proof}

\subsection{Surfaces, Ribbons, and Legendrian Realization}

Here we present versions of Legendrian Realization for graphs, cf.\ \cite[Theorem 2.1]{HKM-tightctctstrfiberedhyp3mfld}.
For this, we say a graph $G$ embedded in a connected surface $S$ with boundary  is {\em non-isolating} if every component of $S-G$ contains a component of $\bdry S$.  Observe that this implies that $G$ is a subgraph of a {\em spine} $\bar{G}$ of $S$, a graph onto which $S$ deformation retracts.
When \cite{HKM-tightctctstrfiberedhyp3mfld} defines `non-isolating' for graphs in convex surfaces, they require that the graph transversally intersects the dividing curves and any univalent leaves must lie on the dividing curves.  However, this constraint on univalent leaves is unnecessary.

Recall that an oriented {\em convex} surface $\Sigma$ in a contact $3$-manifold $(M,\xi)$ 
has a dividing set $\Gamma$, an (isotopy class of) embedded multicurve that chops $\Sigma$ into surfaces $\Sigma_+$ and $\Sigma_-$ so that all tangencies with $\xi$ of sign $\pm$ are in $\Sigma_\pm$.   Furthermore there is 
 a vector field $v$ transverse to $\Sigma$ whose flow preserves $\xi$.  

Any singular foliation $\calF$ on $\Sigma$ is also said to be {\em divided by} the multicurve $\Gamma$ if there is an $I$-invariant contact structure $\xi'$ on $\Sigma \times I$ where $\calF = \chi\vert_{\Sigma \times \{0\}}$.
By Giroux Flexibility \cite{girouxconvex} (see also \cite[Theorem 4.8.11]{GeiBook} or \cite[Theorem 3.4]{honda-classificationI}), if $\calF$ is divided by $\Gamma$, then there is an isotopy $\phi_t$ for $t \in [0,1]$ of $\Sigma$ in a neighborhood of $\Sigma$ that fixes $\Gamma$ so that $\phi_0(\Sigma)=0$, $\xi\vert_{\phi_1(\Sigma)} = \calF$, and $\phi_t(\Sigma)$ is transverse to $v$ for all $t$.

\begin{lemma}\label{lem:legendrianrealization}
    Let $\Sigma$ be a closed convex surface with dividing set $\Gamma$. Then for any choice of spines $G$ of the components of $\Sigma \cut \Gamma$, $\Sigma$ is isotopic rel-$\Gamma$ through convex surfaces to a surface $\Sigma'$ so that the union of the spines $G$ is now a Legendrian graph $\Lambda'$ and each component $\Sigma'_0$ of $\Sigma' \cut \Gamma$ is a ribbon of the Legendrian graph $\Lambda'_0=\Lambda' \cap \Sigma'_0$.
\end{lemma}

\begin{proof}
    This is basically just an application to the spines of $\Sigma \cut \Gamma$ of the Legendrian Realization for graphs \cite[Theorem 2.1]{HKM-tightctctstrfiberedhyp3mfld} (a straightforward extension of Legendrian realization for curves \cite[Theorem 3.7]{honda-classificationI} which itself is a consequence of Giroux Flexibility).  The only thing to observe is that a foliation indeed exists on $\Sigma$ realizable as a characteristic foliation $\calF$ that has singular set equal to $G$, where the flow lines are the union of spanning arcs of the annuli of $\Sigma \cut G$ flowing transverse to $\Gamma$ from $G \cap \Sigma_+$ to $G \cap \Sigma_-$. Having isotoped $\Sigma$ rel-$\Gamma$ to $\Sigma'$ so that $\xi\vert_{\Sigma'} = \calF$, $G$ is realized as the Legendrian graph $\Lambda'$ and each component of $\Sigma' \cut \Gamma$ is a ribbon of its component of $\Lambda'$.
\end{proof}

\begin{lemma}\label{lem:legspinesofribbons}
    Let $R$ be the ribbon of a Legendrian graph.  Then any spine of $R$ may be realized as a Legendrian graph $\Lambda$ so that $R(\Lambda)$ is isotopic to $R$ rel-$\bdry$.
\end{lemma}
\begin{proof} 
Suppose $R$ is the ribbon $R(\Lambda_0)$.  Then $S = \bdry N(\Lambda_0)$ is a convex surface divided by $\bdry R$ and $R$ is isotopic to $S_+$ rel-$\bdry R$. Given a spine of $R$, let $G$ be its image in $S_+$ under the isotopy.  The isotopy  of Lemma~\ref{lem:legendrianrealization} takes $S$ rel-$\bdry R$ to a convex surface $S'$ in which $G$ is now a Legendrian graph $\Lambda$ so that $S'_+$ is the ribbon $R(\Lambda)$.
\end{proof}

\begin{lemma}\label{lem:LegRealizationGraphs}
Given an open book supporting the contact $3$-manifold $(M,\xi)$ and containing a non-isolating connected graph $G$ in the interior of a page,
the open book may be isotoped so that it supports $(M,\xi)$ with the graph $G$ now Legendrian in its page. 

Conversely, for any connected Legendrian graph $\Lambda$ in a closed contact $3$-manifold $(M,\xi)$, there is a supporting open book that contains $\Lambda$ as a non-isolating graph in a page.
\end{lemma}

Etnyre also discusses the Legendrian realization of a spine of a page of an open book in \cite[Remark 30]{Etnyre-LecturesOpenBooks} and thereabouts.

\begin{proof}
By \cite[Lemma 2.5]{LevQPDecompLagrang}, using Legendrian Realization for graphs \cite[Theorem 2.1]{HKM-tightctctstrfiberedhyp3mfld}, if $G$ is a non-isolating graph in a page $S$ of an open book supporting a contact structure $(M, \xi)$, then $S$ may be isotoped to make $G$ Legendrian.  In particular $G$ extends to a spine $\bar{G}$ of $S$, and $S$ may be isotoped to make the spine Legendrian so that $S$ is a ribbon, cf.\ Lemma~\ref{lem:legspinesofribbons}.
Indeed, as the above isotopy is induced from the Legendrian realization of this spine as a graph in the closed convex surface made from two pages (where the dividing set corresponds to the binding $\bdry S$) the entire open book may be isotoped so that it supports $(M,\xi)$ while the spine $\bar{G}$ is Legendrian and the page $S$ is a ribbon.  Since $G$ is now a Legendrian subgraph of the spine $\bar{G}$, its Legendrian framing is the framing by $S$. 
Moreover there is a regular neighborhood of $G$ in the interior of $S$ that is nearly a ribbon of $G$, only needing slight modification it encounters $\bar{G}-G$.  (A slight isotopy of the interior of $S$ would allow for the ribbon of $G$ to be contained in $S$ at the expense of maintaining that $S$ is a ribbon of $\bar{G}$.)

On the other hand, given a connected Legendrian graph $\Lambda$ in a closed contact $3$-manifold $(M,\xi)$, its exterior $(M_\Lambda, \xi_\Lambda)$ is a contact $3$-manifold with convex boundary. As such, there is a partial open book supporting $(M_\Lambda, \xi_\Lambda)$.  Yet since the exterior $(M_\Lambda, \xi_\Lambda)$ may also be viewed as the exterior of a ribbon $R(\Lambda)$, the partial open book extends to an open book supporting $(M,\xi)$ that contains $\Lambda$ and $R(\Lambda)$ in a page.  Furthermore, by construction of the partial open books, $\Lambda$ is necessarily non-isolating in its page.
\end{proof}

\begin{remark}
 By \cite[Theorem 1.8]{LevQPDecompLagrang}, for any Seifert surface $S$ of a link in a $3$-manifold $Y$, there is a contact structure on $Y$ so that $S$ is the ribbon of some Legendrian graph.  However the Legendrian graph may have overtwisted complement. 
\end{remark}

\subsection{Proof of main theorem}

\begin{theorem}\label{thm:main}
    Let $\Lambda$ be a connected Legendrian graph in a closed contact $3$-manifold.  If $\Lambda$ has the Tight Reattachment Property,
    then the transverse link $\bdry R(\Lambda)$ is non-loose.
\end{theorem}

\begin{proof}
    Say $\Lambda$ is a connected Legendrian graph in $(M,\xi)$.
    By Lemma~\ref{lem:LegRealizationGraphs}, there is an open book $(S,\phi)$ supporting $(M,\xi)$ that contains $\Lambda$ as a non-isolating Legendrian graph in a page.
    In particular,  the exterior $(M_\Lambda, \Gamma_\Lambda, \xi_\Lambda)$ of $\Lambda$ is supported by a partial open book $(S,P,\phi_P)$
    that extends to an open book $(S, \phi)$ for $(M,\xi)$.  Here, $P$ is a subsurface of $S$ and $\phi_P \colon P \to S$ is a homeomorphism to its image that is the identity on $\bdry P \cap \bdry S$ that extends across $S\cut P$ to the homeomorphism $\phi \colon S \to S$.

    By the Tight Reattachment Property there is a map $\psi \colon S \to S$ with support in $S \cut P$ so that the open book $(S,\phi\psi)$ supports a tight contact structure.  Setting $h=\phi\psi$ and $g= \psi^{-1}$ so that $hg=\phi$, an application of Theorem~\ref{thm:tightfactorization} below then further shows that the transverse link $\bdry R(\Lambda)$ is also non-loose.    
\end{proof}

\begin{theorem}\label{thm:tightfactorization}
    Let $\Lambda$ be a connected Legendrian graph in $(M_{S,hg},\xi_{S,hg})$ that is non-isolating in a page of the supporting open book $(S,hg)$.  
    Suppose $(M_{S,h},\xi_{S,h})$ is tight and the support of $g$ is contained in a regular neighborhood of $\Lambda$ in $S$.
    Then the transverse link $\bdry R(\Lambda)$ has tight complement. 
\end{theorem}

\begin{proof} 
    First let $F$ be a regular neighborhood of $\Lambda$ in $S$ that contains the support of $g$. Since $\Lambda$ is connected, $F$ is connected. Then observe that the binding of the open book $(F, g \vert_{F})$ supporting $(M_{F, g \vert_{F}}, \xi_{F, g \vert_{F}})$ is a transverse link with tight complement (see proof of \cite[Lemma 3.1]{EtnyreVelaVick}).  
    Furthermore, the supported contact structure $\xi_{F, g \vert_{F}}$ may be taken so that any chosen spine of any page is a Legendrian graph for which the page is a ribbon.  As such, we may regard the page $F \times \{1/2\}$ in $(M_{F, g \vert_{F}}, \xi_{F, g \vert_{F}})$ as the ribbon of a Legendrian graph $\Lambda_F$.
    Letting $R(\Lambda_F)$ be a thinner ribbon contained in the interior of the page $F \times \{1/2\}$, the transverse link $\bdry R(\Lambda_F)$ is transversally isotopic to the binding of $(F, g \vert_{F})$.  Thus $\bdry R(\Lambda_F)$ is a transverse link with tight complement in the interior of the page $F \times \{1/2\}$.

    Since $g \vert_{S \cut F}$ is trivial and $F$ is non-isolating, 
    the extension of $(F, g\vert_{F})$ to $(S,g)$ may be viewed as a Murasugi sum of an open book $(S', id)$ with trivial monodromy upon $(F, g \vert_{F})$.   
    Moreover this Murasugi sum may be done within a neighborhood of the page $F \times \{0\}$ so that it is disjoint from $\Lambda_F$ and the copy $R(\Lambda_F)$ of $F$ that it bounds in $F \times \{1/2\}$. So after this Murasugi sum, 
    we have a natural inclusion $F \times \{1/2\} \hookrightarrow S \times \{1/2\}$ giving an inclusion of the Legendrian graph $\Lambda_F$, its ribbon $R(\Lambda_F)$, and transverse boundary $\bdry R(\Lambda_F)$ into the page $S \times \{1/2\}$ of the supported contact manifold $(M_{S,g}, \xi_{S,g})$.  We denote their images as $\Lambda_g$, $R(\Lambda_g)$, and $\bdry R(\Lambda_g)$.   
    The complement of $\bdry R(\Lambda_g)$ in $(M_{S,g}, \xi_{S,g})$ is then the connected sum of  $(M_{S',id}, \xi_{S',id})$ and the tight complement of the transverse link $\bdry R(\Lambda_F)$ in $(M_{F, g \vert_{F}}, \xi_{F, g \vert_{F}})$.
    Since $(M_{S',id}, \xi_{S',id})$ is a connected sum of several copies of $(S^1 \times S^2, \xi_{std})$, it is tight.    
    Since the two summands are tight, the complement of $\bdry R(\Lambda_g)$ tight too by Lemma~\ref{lem:connectedsum}.
    
    Now consider the contact connected sum of $(M_{S,h},\xi_{S,h})$ with $(M_{S,g},\xi_{S,g})$ along a neighborhood of a point in the interior of the page $S \times \{0\}$ of each open book. Since $\Lambda_g$ and $R(\Lambda_g)$ are contained in the page $S\times \{1/2\}$ of its open book, they are disjoint from the summing sphere.  So let $\Lambda_{h\#g}$ be the Legendrian graph that is the image of $\Lambda_g$ in this sum and let $R(\Lambda_{h\#g})$ be the corresponding ribbon.  Since $(M_{S,h},\xi_{S,h})$ is tight by hypothesis and the complement of $\bdry R(\Lambda_g)$ is tight, we now have that the transverse link $\bdry R(\Lambda_{h\#g})$ has tight complement  by Lemma~\ref{lem:connectedsum} as well.
    
    Finally, Baldwin \cite{baldwin-steincobordisms} exhibited a Legendrian link $\mathbb{L}$ in $(M_{S,hg},\xi_{S,hg})$ upon which $(+1)$ contact surgery yields the connected sum  $(M_{S,h},\xi_{S,h})\#(M_{S,g},\xi_{S,g})$.  In fact, the surgery dual link $\mathbb{L}^*$ may be regarded as a Legendrian link in (a neighborhood of) the connected sum of the two  pages $S\times \{0\}$ from each summand so that $(-1)$ contact surgery on $\mathbb{L}^*$ yields $(M_{S,hg},\xi_{S,hg})$. In particular, for $t \in (\epsilon, 1-\epsilon)$ with $0 < \epsilon < 1/2$, the pages $S \times \{t\}$ of $(M_{S,h},\xi_{S,h})$ are identified with the the pages $S \times \{t/2\}$ of $(M_{S,hg},\xi_{S,hg})$, and the pages $S \times \{t\}$ of $(M_{S,g},\xi_{S,g})$ are identified with the the pages $S \times \{t/2+1/2\}$ of $(M_{S,hg},\xi_{S,hg})$.  Thus after the $(-1)$ contact surgery on $\mathbb{L}^*$, the Legendrian graph $\Lambda_{h\#g}$ and ribbon $R(\Lambda_{h\#g})$ become the Legendrian graph $\Lambda_{hg}$ and ribbon $R(\Lambda_{hg})$
    in the page $S \times \{3/4\}$.  As the complement of $\bdry R(\Lambda_{hg})$ is obtained by $(-1)$ contact surgery on the link $\mathbb{L}^*$ in the tight complement of $\bdry R(\Lambda_{h\#g})$, we would like to say that this implies it is tight as well.
    
    While Wand shows that $(-1)$ contact surgery on a closed manifold preserves tightness \cite{wand}, the complement of $\bdry R(\Lambda_{h\#g})$ is not closed.  Nevertheless, we may adapt the above construction so that the complement of $\bdry R(\Lambda_{h\#g})$ is embedded in a closed tight manifold.  Then by \cite{wand}, the $(-1)$ contact surgery on $\mathbb{L}^*$ now in this closed manifold will yield a closed tight manifold in which the complement of $\bdry R(\Lambda_{hg})$ is embedded.

    To do this, we first observe that, from the argument in the 2nd paragraph of the proof of \cite[Theorem 1.6]{Etnyre-knotsinOTctctstr}, the contact complement of our initial transverse link as the binding $B$ of an open book $(F,g\vert_F)$ embeds in a universally tight contact manifold as a component of the complement of pre-Lagrangian tori $T$.  Let $N'$ be the union of the complement of $B$ and these tori $T$, the `closed complement' of $B$. Hence $N'$ is universally tight and contains the complement of $B$ as its interior.

    Let $N''$ be a second copy of $N'$ (or other similarly constructed universally tight manifold with incompressible boundary) and glue it to the original $N'$ so that the contact structures match up along collars of $\bdry N'$ and $\bdry N''$.  Colin shows that this resulting manifold is also universally tight \cite{colin}.  For short, we may regard this as replacing a neighborhood of $B$ by $N''$.  Carry this replacement through the above construction to where we eventually have a neighborhood of $\bdry R(\Lambda_{h\#g})$ replaced by $N''$.  The construction now gives the exterior of $\bdry R(\Lambda_{h\#g})$ embedded in a tight manifold as desired.  Thus we may apply \cite{wand} to obtain the exterior of $\bdry R(\Lambda_{hg})$ embedded in a tight manifold.  Hence the transverse link $\bdry R(\Lambda_{hg})$ is tight.
\end{proof}

\begin{remark}
An alternative approach would be to use the (Stein) cobordism from $(M_{S,h},\xi_{S,h}) \sqcup (M_{S,g}, \xi_{S,g})$ to $(M_{S,hg}, \xi_{S,hg})$ of \cite[Section 8.2]{BEVHM} instead of \cite{baldwin-steincobordisms}.  However, the exposition would take a bit more clarity on the required Hopf deplumbings involved.  
\end{remark}

\subsection{Fusions of Legendrian graphs and Murasugi sums}\label{sec:fusion}

Given the standard contact structure $\xi =\ker (dz-y\,dx)$ on $\R^3$, let $\xi^n$ be the contact structure on $\R^3$ induced from the $n$-fold cyclic branched cover along the $z$-axis for each positive integer $n$.  Observe that $(\R^3, \xi^n)$ is in fact contactomorphic to $(\R^3,\xi)$. Let $\lambda$ be the non-negative $x$-axis and let $\lambda^n$ be its lift to $(\R^3,\xi^n)$.   Any valence $n$ vertex of a Legendrian graph has a neighborhood contactomorphic to $(\R^3, \xi^n,\lambda^n)$.

Next, let $\lambda_+$ be a Legendrian arc starting at the origin whose front projection continues along the positive $x$-axis to $x=1$ and thereafter has positive slope. Similarly, let $\lambda_-$ be a Legendrian arc starting at the origin whose front projection continues along the negative $x$-axis to $x=-1$ and thereafter has negative slope.  Note that  $\lambda_+$ is Legendrian isotopic to $\lambda$ rel-$\bdry \lambda$ and we may choose $\lambda_-$ to be a reflection of $\lambda_+$ through the origin.
Define $\lambda^n_+$ and $\lambda^n_-$ similarly to $\lambda^n$ via cyclic branched covers along the $z$-axis.  Hence any valence $n$ vertex of a Legendrian graph has a neighborhood contactomorphic to each $(\R^3, \xi^n,\lambda^n_+)$ and $(\R^3, \xi^n,\lambda^n_-)$.

\begin{defn}[Fusions of Legendrian graphs]
Suppose for each $i=+,-$ that $v_i$ is a valence $n$ vertex of a Legendrian graph $\Lambda_i$ in $(M_i,\xi_i)$.  Choose a contactomorphism $\phi_i$ of a neighborhood $B_i$ of $\Lambda_i$ at $v_i$ with $(\R^3, \xi^n,\lambda^n_i)$.
The  {\em Legendrian fusion (of order $n$)} $\Lambda_+ \ast \Lambda_-$ of $\Lambda_+$ and $\Lambda_-$ along vertices $v_+$ and $v_-$ with respect to these chosen contactomorphisms $\phi_+$ and $\phi_-$ is defined follows.
Form the connected sum $(M_+ \# M_-, \xi_+ \# \xi_-)$ by identifying the complements of $\phi_+^{-1}(\R^3\vert_{z<0}, \xi^n,\lambda^n_+)$ and $\phi_-^{-1}(\R^3\vert_{z>0}, \xi^n,\lambda^n_-)$ along their boundaries so that 
$\phi_+^{-1}(\R^3\vert_{z\geq0}, \xi^n,\lambda^n_i) \cup \phi_-^{-1}(\R^3\vert_{z\leq0}, \xi^n,\lambda^n_i)$ is contactomorphic to $(\R^3, \xi^n,\lambda^n_+ \cup \lambda^n_-)$.  From the inclusions of $(M_+, \xi_+)-\phi_+^{-1}(\R^3\vert_{z<0}, \xi^n,\lambda^n_+)$ and $(M_, \xi_-)- \phi_-^{-1}(\R^3\vert_{z>0}, \xi^n,\lambda^n_-)$ into $(M_+ \# M_-, \xi_+ \# \xi_-)$, we may regard $\Lambda_+$ and $\Lambda_-$ as Legendrian graphs in $(M_+ \# M_-, \xi_+ \# \xi_-)$ whose intersection is a single point $v$, the image of the points $v_+$ and $v_-$.  Then the fusion $\Lambda_+ \ast \Lambda_-$ is the union of $\Lambda_+$ with $\Lambda_-$ in $(M_+ \# M_-, \xi_+ \# \xi_-)$.  Figure~\ref{fig:LegendrianFusion}
illustrates this Legendrian fusion operation.
\end{defn}

\begin{figure}
    \centering
    \includegraphics[width=.7\textwidth]{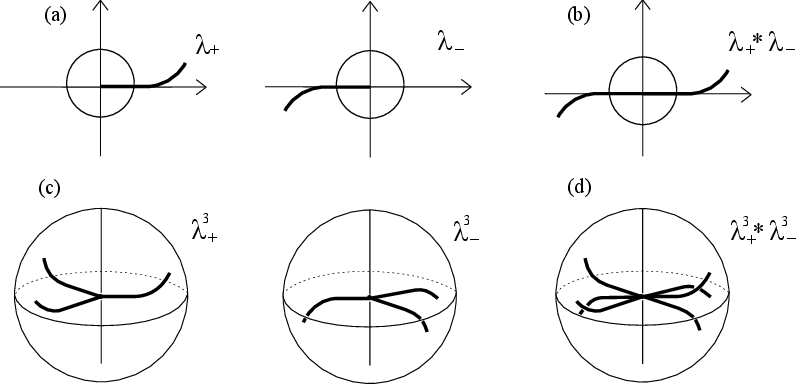}
    \caption{(a) The front projection of the Legendrian arcs $\lambda_+$, and $\lambda_-$.  (b) The Legendrian fusion $\lambda_+ \ast \lambda_-$. (c) Legendrian graphs $\lambda_+^n$ and $\lambda_-^n$  for $n=3$.  (d) The Legendrian fusion $\lambda_+^n \ast \lambda_-^n$ for $n=3$.}
    \label{fig:LegendrianFusion}
\end{figure}

\begin{defn}[Murasugi sums of ribbons of Legendrian graphs]
At the point of fusion, $\Lambda_+ \ast \Lambda_-$ is locally contactomorphic to $\lambda_+^n \ast \lambda_-^n$ which is the $n$-fold branched cover along the $z$-axis of $\lambda_+ \ast \lambda_- = \lambda_+ \cup \lambda_-$ in $(\R^3,\xi)$.  We may choose ribbons $R(\lambda_+)$, $R(\lambda_-)$, and $R(\lambda_+ \ast \lambda_-)$
so that for some ball neighborhood $B$ of the origin, $R(\lambda_+)-B$ and $R(\lambda_-)-B$ are disjoint with union equalling $R(\lambda_+ \ast \lambda_-)-B$. Observe that within $B$, the surfaces $R(\lambda_+)$ and $R(\lambda_-)$ form a clasp intersection and project to subsurfaces of $R(\lambda_+ \ast \lambda_-) \cap B$. Nevertheless, we may regard $R(\lambda_+ \ast \lambda_-)$ as obtained from trimming $R(\lambda_+)$ and $R(\lambda_-)$ near their vertices and then joining to form $R(\lambda_+ \ast \lambda_-)$. Topologically, we obtain $R(\lambda_+ \ast \lambda_-)$ as a $2$-Murasugi sum of $R(\lambda_+)$ and $R(\lambda_-)$, and we define this to be what's meant by the $2$-Murasugi sum of the ribbons $R(\lambda_+)$ and $R(\lambda_-)$. Figure~\ref{fig:RibbonSums}(Top) shows this operation.  As $R(\lambda^n_+\ast \lambda^n_-)$ is the $n$-fold branched cover of $R(\lambda_+ \ast \lambda_-)$, we define $R(\lambda^n_+\ast \lambda^n_-)$ to be the {\em $2n$-Murasugi sum} of the ribbons $R(\lambda^n_+)$ and $R(\lambda^n_-)$, shown in Figure~\ref{fig:RibbonSums}(Bottom) for $n=3$.   For the topological definition of a Murasugi sum of Seifert surfaces, see \cite{gabai-murasugi} for example.  Thus we immediately have the following.
\begin{lemma}\label{lem:ribbonmurasugi}
        Let $\Lambda = \Lambda_+ \ast \Lambda_-$ be a Legendrian fusion of $\Lambda_+$ and $\Lambda_-$.
    Then $R(\Lambda)$ is a Murasugi sum of $R(\Lambda_+)$ and $R(\Lambda_-)$. \qed
\end{lemma}
\end{defn}

\begin{figure}
    \centering
    \includegraphics[width=.9\textwidth]{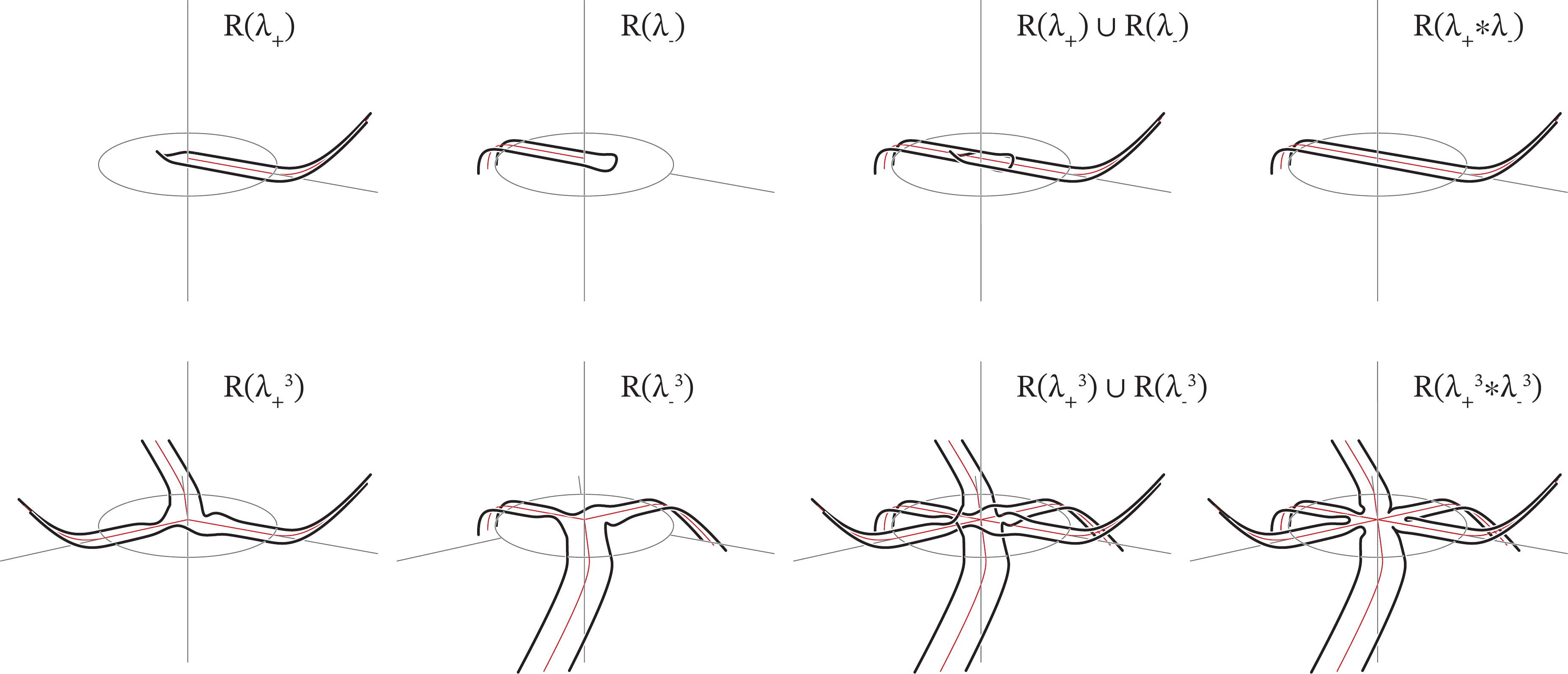}
    \caption{  (Top) Ribbons of $\lambda_+$ and $\lambda_+$, the union of these ribbons $R(\lambda_+) \cup R(\lambda_-)$ and the ribbon of their fusion $R(\lambda_+ \ast \lambda_-)$ are illustrated in the standard $(\R^3,\xi)$.  Also shown are the $z$-axis, a circle in the $xy$-plane, and a porition of the positive $x$-axis outside the circle.   (Bottom) The above, but with $\lambda_+^n$ and $\lambda_-^n$ for $n=3$, in the cyclic branched cover $(\R^3,\xi^n)$.  \\
    ``The ribbon of a fusion is a Murasugi sum of ribbons.''}
    \label{fig:RibbonSums}
\end{figure}

Note that we cannot have either $R(\lambda_+)$ or $R(\lambda_-)$ as a subsurface of $R(\lambda_+ \ast \lambda_-)$. Were, say $R(\lambda_+)$ contained in $R(\lambda_+ \ast \lambda_-)$, then generically the portion of $R(\lambda_+)$ that flows to its vertex (at the origin) would flow to other points of $\lambda_-$ in $R(\lambda_+ \ast \lambda_-)$. However, this would contradict the non-integrability of the contact structure.

\begin{theorem}\label{thm:murasugisums}
    Suppose $\Lambda$ is the Legendrian fusion $\Lambda_+ \ast \Lambda_-$ of two Legendrian graphs.
    If each $\Lambda_+$ and $\Lambda_-$ have the Tight Reattachment Property, then $\Lambda$ has the Tight Reattachment Property.
\end{theorem}

\begin{proof}
    By Lemma~\ref{lem:ribbonmurasugi}, $R(\Lambda)$ is a Murasugi sum of $R(\Lambda_+)$ and $R(\Lambda_-)$.
    Indeed, the natural inclusions of $\Lambda_+$ and $\Lambda_-$ into $M_+\# M_-$ extend to inclusions of their ribbons $R(\Lambda_+)$ and $R(\Lambda_-)$ as {\em nearly} subsurfaces of the ribbon $R(\Lambda) = R(\Lambda_+ \ast \Lambda_-)$.

    Let $N(\Lambda_i)$ be a thickening of $R(\Lambda_i)$ so that $\bdry N(\Lambda_i)$ is convex with dividing curves $\bdry R(\Lambda_i)$.  
    Again, the natural inclusions allow us to regard $N(\Lambda_i)$ as also a thickening of $\Lambda_i$ in $M_+ \# M_-$.  These can further be arranged so that $N(\Lambda_+)$ and $N(\Lambda_-)$ intersect within a ball neighborhood of the fusion vertex of $\Lambda = \Lambda_+ \cap \Lambda_-$
    so that the components $\bdry_\pm N(\Lambda_+)$ and the components $\bdry_\pm N(\Lambda_-)$ are naturally identified with subsurfaces of  $\bdry_\pm N(\Lambda)$.

    Suppose for each $i=+,-$ that $\Lambda_i$ has the TRP.  Then choose reattaching maps $\psi_i\colon \bdry_-N(\Lambda_i) \to \bdry_+N(\Lambda_i)$  so that $(M_{\Lambda_i,\psi_i}, \xi_{\Lambda_i,\psi_i})$, the resulting reattachment along $\Lambda_i$, is a tight manifold. With the natural inclusions of $N(\Lambda_i)$ into $M_+ \# M_-$, we may regard $\psi_i$ as a map from $\bdry_-N(\Lambda_i)$ to $\bdry_+N(\Lambda_i)$ there.  
    Then, after the identifications of $\bdry_\pm N(\Lambda_i)$ with subsurfaces of $\bdry_\pm N(\Lambda)$, let $\hat{\psi}_i\colon \bdry_-N(\Lambda) \to \bdry_+N(\Lambda)$ be the extension of $\psi_i$ by the identity. 
    Now define $\psi \colon \bdry_-N(\Lambda) \to \bdry_+N(\Lambda)$ as $\hat{\psi}_+ \circ \hat{\psi}_-$.  We then observe that the rettachement along $\Lambda$ by $\psi$ is the connected sum of the rettachments along $\Lambda_i$ by $\psi_i$ for $i=+,-$.  Since the connected sum of tight manifolds is tight, this shows that $\Lambda$ has the TRP.
\end{proof}

\subsection{Subgraphs and the TRP}

In general, non-looseness of a subset is conferred to any other subset that contains it.  If $T$ is a transverse link containing a a non-loose sub-link $T'$, then $T$ is non-loose as well.  Similarly, if $\Lambda$ is a Legendrian graph containing a non-loose subgraph $\Lambda'$, then $\Lambda$ is non-loose.  Here we observe that the TRP is also passed upwards from $\Lambda'$ to $\Lambda$.  With Theorem~\ref{thm:main}, this shows how this stronger sense of non-looseness of $T'=\bdry R(\Lambda')$ is passed to the non-looseness of $T=\bdry R(\Lambda)$ even though $T'$ is not contained in $T$.

\begin{theorem}~\label{thm:subgraphTRP}
    Let $\Lambda'$ be a connected subgraph of a connected Legendrian graph $\Lambda$ in a closed contact $3$-manifold.  If $\Lambda'$ has the TRP, then $\Lambda$ does too.
\end{theorem}

\begin{proof}
    Suppose $\Lambda$ is a connected Legendrian graph in $(M,\xi)$.
    Since $\Lambda'$ is a subgraph of $\Lambda$, the components $\bdry_\pm N(\Lambda')$ are naturally identified with subsurfaces of $\bdry_\pm N(\Lambda)$.
    Since $\Lambda'$ has the TRP, choose reattaching map $\psi' \colon \bdry_-N(\Lambda') \to \bdry_+N(\Lambda')$ so that the resulting attachment is a tight manifold $(M_{\Lambda',\psi'})$.  With the natural identifications, $\psi'$ extends by the identity to an attaching map $\psi \colon \bdry_-N(\Lambda) \to \bdry_+N(\Lambda)$.  Since $\psi$ is an extension of $\psi'$ by the identity, the reattached manifold $(M_{\Lambda, \psi}, \xi_{\Lambda, \psi})$ is just the tight manifold $(M_{\Lambda', \psi'}, \xi_{\Lambda', \psi'})$.
\end{proof}

\subsection{Comultiplication and the LOSS invariant}
\label{sec:LOSS}

Theorem~\ref{thm:tightfactorization} developed from an adaptation of Baldwin's comultiplication of the Heegaard Floer contact invariant \cite{baldwin-comultiplicity} to a sort of comultiplication for the LOSS invariant \cite{LOSS} given in Theorem~\ref{thm:LOSScomultiplication} below.  One may view Proposition~\ref{prop:ctctinvtreattachment} as an algebraic version of Theorem~\ref{thm:main} where a transverse knot has non-trivial LOSS invariant if a reattachment along a Seifert surface yields a manifold with non-trivial contact invariant.  Since this may be of independent interest, we keep it here.

\bigskip

Let $(Y_{S,h}, \xi_{S,h})$ be the contact $3$-manifold determined by the open book $(S,h)$.  A choice of basis of arcs for $S$ determines a chain complex $\widehat{CF}(-Y_{S,h})$ and an special element ${\bf x}_{S,h} \in \widehat{CF}(-Y_{S,h})$.  The homology class of this element is the Ozsvath-Szabo contact invariant $c(S,h) \in \widehat{HF}(-Y_{S,h})$  of this contact manifold \cite{OzSz-HFctct}.  Notably, if $(Y_{S,h}, \xi_{S,h})$ contains an overtwisted disk, then $c(S,h)=0$. Hence $c(S,h) \neq 0$ implies $(Y_{S,h}, \xi_{S,h})$ is tight.

A non-separating oriented simple closed curve $K$ in the page $S$ determines an oriented Legendrian knot $K_h$ in $(Y_{S,h}, \xi_{S,h})$. If the basis of arcs for $S$ is chosen to intersect $K$ exactly once, the chain complex $\widehat{CF}(-Y_{S,h})$ is adapted to $K_h$. 
When $K_h$ is null-homologous in $Y_{S,h}$, it provides a filtration of $\widehat{CF}(-Y_{S,h})$ yielding a filtered complex $\widehat{CFK}(-Y_{S,h},K_h)$ and associated homology $\widehat{HFK}(-Y_{S,h},K_h)$.
The LOSS invariant of $K_h$ may be regarded as the homology class $\widehat{\mathcal{L}}(S,h,K) \in \widehat{HFK}(-Y_{S,h},K_h)$ of the special element ${\bf x}_{S,h}$ \cite{LOSS}. Notably, if the complement of $K_h$ contains an overtwisted disk, then $\widehat{\mathcal{L}}(S,h,K)=0$.  So if $\widehat{\mathcal{L}}(S,h,K)\neq 0$, then the complement of $K_h$ is tight.  Accordingly if $c(S,h) \neq 0$, then it follows that $\widehat{\mathcal{L}}(S,h,K) \neq 0$.  

(More accurately, the LOSS invariant is an equivalence class $\widehat{\mathcal{L}}(K_h)$ of pairs $(\widehat{HFK}(-Y_{S,h},K_h),[{\bf x}_{S,h}])$ up to isomorphism.  However, as our interest is in determining tightness, here we only concern ourselves with whether or not the invariant is trivial.)

\medskip

Given an open book $(S,hg)$ whose monodromy is expressed as a composition,
Baldwin \cite{baldwin-comultiplicity} defines a comultiplication map on the contact invariant
\[ \widetilde{\mu} \colon \widehat{HF}(-Y_{S,hg}) \to \widehat{HF}(-Y_{S,h}) \otimes_{\Z_2} \widehat{HF}(-Y_{S,g})\]
that sends $c(S,hg)$ to $c(S,h) \otimes_{\Z_2} c(S,g)$.  
So if $c(S,h)\neq 0$ and $c(S,g)\neq 0$ then $c(S,hg)\neq 0$. 

Here we extend this to a kind of comultiplicity for the LOSS invariant.

\begin{theorem}\label{thm:LOSScomultiplication}
    Suppose $K$ is a non-separating oriented simple close curve in the surface $S$. If the knots $K_{hg}$ and $K_g$ are null-homologous in their manifolds $Y_{S,hg}$ and $Y_{S,g}$, then there is a comultiplication map
    \[ \widetilde{\mu} \colon \widehat{HFK}(-Y_{S,hg},K_{hg}) \to \widehat{HF}(-Y_{S,h}) \otimes_{\Z_2} \widehat{HFK}(-Y_{S,g},K_g)\]
    that sends $\widehat{\mathcal{L}}(S,hg,K)$ to $c(S,h) \otimes_{\Z_2} \widehat{\mathcal{L}}(S,g,K)$.

    In particular, if $c(S,h) \neq 0$ and $ \widehat{\mathcal{L}}(S,g,K) \neq 0$ then $\widehat{\mathcal{L}}(S,hg,K) \neq 0$.
\end{theorem}

\begin{proof}
    Baldwin shows that his map $\widetilde{\mu}$ is induced from a map
    \[\mu  \colon \widehat{CF}(-Y_{S,hg}) \to \widehat{CF}(-Y_{S,h}) \otimes_{\Z_2} \widehat{CF}(-Y_{S,g})\]
    in which 
    \[\mu({\bf x}_{S,hg}) = {\bf x}_{S,h} \otimes_{\Z_2} {\bf x}_{S,g}.\]
    Along with the basepoint $z$ used to define these chain complexes and map, our inclusion of the oriented curve $K$ in $S$ and the corresponding knots adds another basepoint $w$ which induces a filtration.  

    From a basis of arcs ${\bf a}=\{a_1, \dots, a_n\}$ in $S$, a perturbation to a basis of arcs ${\bf b} = \{b_1, \dots, b_n\}$ where each endpoint of $\bdry b_i$ moves a bit in the direction of $\bdry S$ and so that $b_i$ intersects $a_i$ exactly once.      Similarly a basis of arcs ${\bf c} = \{c_1, \dots, c_n\}$ is obtained by perturbing ${\bf b}$.  Together $a_i, b_i, c_i$ appear as in Figure~\ref{fig:arcsabc} in a strip neighborhood of $a_i$. The basepoint $z$ is placed outside these neighborhoods in $S$ as also shown. 
    
    \begin{figure}
    \centering
    \includegraphics[width=.15\textwidth]{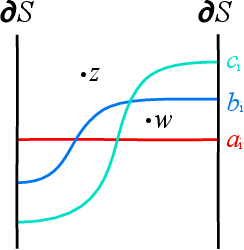}
    \caption{The placement of the triple of arcs $ai,b_i,c_i$ and the basepoints $z$ and $w$.}
    \label{fig:arcsabc}
\end{figure}

    For the LOSS invariant, the arcs ${\bf a}$ are also chosen so that $K$ is disjoint from all $a_i$ except $a_1$ which it transversally intersects exactly once. Then, along with $z$ as above, the basepoint $w$ is chosen disjoint from ${\bf a} \cup {\bf b} \cup {\bf c}$ in the strip neighborhood of $a_1$ so that it is 
    \begin{itemize}
        \item between $a_1$ and $b_1$ to define $K_g \subset Y_{S,g}$ in the page of $(S, g)$ with respect to ${\bf a}$ and ${\bf b}$, and also
        \item between $a_1$ and $c_1$ to define $K_{hg} \subset Y_{S,hg}$ in the page of $(S,hg)$ with respect to ${\bf a}$ and ${\bf c}$.   
    \end{itemize} 
    This however forces $w$ to be in the same region as $z$ making in a trivial curve in the page of $(S,h)$ with respect to ${\bf b}$ and ${\bf c}$ so that the induced knot $K_h \subset Y_h$ is a local, trivial unknot.  

    For each $f \in \{h,g,hg\}$, these basepoints filter the complex $\widehat{CF}(-Y_{S,f})$ to produce   the filtered complex $\widehat{CFK}(-Y_{S,f}, K_{f})$ whereupon the homology class of ${\bf x}_f$ in $\widehat{HFK}(-Y_{S,f}, K_{f})$ becomes the LOSS invariant of $K_f$.  Of course, in the case of $K_h$, the basepoint $w$ provides no extra filtering, so we have that $\widehat{CFK}(-Y_{S,h}, K_{h}) \cong \widehat{CF}(-Y_{S,h})$ and the LOSS invariant of $K_h$ is equivalent to $c(S,h)$.

    The arguments of Baldwin extend to incorporate the filtration induced by the extra basepoint, allowing $\mu$ to become a map
    \[\mu  \colon \widehat{CFK}(-Y_{S,hg},K_{hg}) \to \widehat{CF}(-Y_{S,h},) \otimes_{\Z_2} \widehat{CFK}(-Y_{S,g},K_g)\]
    that still takes ${\bf x}_{S,hg}$ to ${\bf x}_{S,h} \otimes_{\Z_2} {\bf x}_{S,g}$.
    Hence $\widetilde{\mu}$ becomes a map 
    \[\widetilde{\mu}  \colon \widehat{HFK}(-Y_{S,hg},K_{hg}) \to \widehat{HF}(-Y_{S,h}) \otimes_{\Z_2} \widehat{HFK}(-Y_{S,g},K_g)\]
    that takes $\widehat{\mathcal{L}}(S,hg,K)$ to $c(S,h) \otimes_{\Z_2} \widehat{\mathcal{L}}(S,g,K)$.
\end{proof}

\begin{prop}\label{prop:ctctinvtreattachment}

Let $K$ be an oriented simple closed curve in the surface $S$ that bounds a subsurface $S_K$ with positive genus.  Suppose that $g$ and $h$ are diffeomorphisms of $S$ which restrict to the identity on $\bdry S$, that $c(S,h) \neq 0$, and that $g$ is supported in $S_K$.

Let $(S', hg\eta)$ be the boundary connected sum of $(S,hg)$ with a positive Hopf band $(H, \eta)$.  Let $K'$ be the result of a single handle-slide of $K$ over the core of the Hopf band so that $K'$ is non-separating in $S'$.
Then $\widehat{\mathcal{L}}(S',hg\eta, K') \neq 0$.  

Furthermore, we may view $S_K$ as the ribbon $R(\Lambda)$ of a Legendrian graph $\Lambda$ so that $K = \bdry R(\Lambda)$ is a transverse knot.  Consequently, the complement of $K$ is tight.
\end{prop}

\begin{proof}
    By Theorem~\ref{thm:LOSScomultiplication}, $\widehat{\mathcal{L}}(S',hg\eta, K') \neq 0$ will follow if $\widehat{\mathcal{L}}(S',g\eta, K') \neq 0$ since $c(S,h)\neq 0$ by assumption.

    The curve $K'$ in $S'$  bounds a subsurface $P \cup S_K$ where $P$ is a pair of pants whose boundary is the two curves $K$ and $K'$ and a component of $\bdry S'$ so that $P$ contains the support of $\eta$.  Recall that $a_1$ is a properly embedded arc in $S'$  intersecting $K'$ exactly once.  Let $N$ be a closed small regular neighborhood of $K' \cup a_1$ (that also contains $b_1$ and $c_1$).  Then in $\bdry N$ is an properly embedded separating arc $\delta$ in $S'$ that is disjoint from $P \cup S_K$.
    Cutting $S'$ along $\delta$ into a surface $S_0$ and the surface $S_{NPK} = N \cup P \cup S_K$ (where $N \cut (P \cup S_K)$ is just an annulus).  Conversely, $S'$ is a boundary connected sum of $S_0$ and $S_{NPK}$.  Hence, as the support of $g\eta$ is contained in $P \cup S_K$, we may view the open book $(S', g\eta)$ containing the curves $K$ and $K'$ as the boundary connected sum of the open book $(S_0, id)$ and the open book $(S_{NPK}, g\eta)$ containing the curves $K$ and $K'$.
    In particular, since the contact connected sum induces an isomorphism 
    \[ \widehat{CF}(-Y_{S', g\eta}) \cong \widehat{CF}(-Y_{S_0,id}) \otimes_{\Z_2} \widehat{CF}(-Y_{S_{NPK},g\eta})\]
    in which ${\bf x}_{S', g\eta}$ is identified with ${\bf x}_{S_0, id} \otimes_{\Z_2} {\bf x}_{S_{NPK},g\eta}$.
    Thus this passes to a map involving the filtered complexes
    \[ \widehat{CFK}(-Y_{S', g\eta},K'_{g\eta}) \to \widehat{CF}(-Y_{S_0,id}) \otimes_{\Z_2} \widehat{CFK}(-Y_{S_{NPK},g\eta},K'_{g\eta})\]
    so that, in homology, $\widehat{\mathcal{L}}(S',g\eta,K')$ is sent to $c(S_0,id) \otimes_{\Z_2} \widehat{\mathcal{L}}(S_{NPK},g\eta,K')$.

    Since $c(S_0,id) \neq 0$, we have that $\widehat{\mathcal{L}}(S',g\eta,K') \neq 0$ if $\widehat{\mathcal{L}}(S_{NPK},g\eta,K')\neq 0$.   To show this, observe that $K$ is the connected binding of the open book $(S_K,g)$, and the open book $(S_{NPK},g\eta)$ contains $K'$ as a Legendrain approximation of this transverse knot $K$.  As the binding of $(S_K,g)$, the transverse LOSS invariant of $K$ is non-zero.  Hence the LOSS invariant of its Legendrian approximation $K'$ is also non-zero. That is, $\widehat{\mathcal{L}}(S_{NPK},g\eta,K')\neq0$ as desired.

    Now having that $\widehat{\mathcal{L}}(S',hg\eta, K') \neq 0$, $K'_{hg\eta}$ is a Legendrian knot in $Y_{S',hg\eta} \cong Y_{S,hg}$ with tight complement.   By construction, it is a Legendrian approximation of the transverse knot $K=\bdry R(\Lambda)$.  Thus $K$ has tight complement as well.
\end{proof}

\section{Examples}
Here we demonstrate basic applications of Theorem~\ref{thm:main} and Theorem~\ref{thm:murasugisums}.  For this, let us establish notation for the contact structures on $S^3$.  Of course we have the unique tight contact structure which we denote $\xi_{std}$.
Up to isotopy, there is an integer family of overtwisted contact structures $\xi_d$ on $S^3$, which are distinguished by their $d_3$-invariant $d_3(\xi_d) = d \in \mathbb{Z} + \frac{1}{2}$; see \cite{DGS} for a general definition.  For one's bearings, the open book of the negative Hopf band supports the contact structure $\xi_{1/2}$ while the open book of the positive Hopf band supports $\xi_{std}$ which has $d_3(\xi_{std}) = -1/2$.

\subsection{Legendrian unknots with the TRP}

\begin{lemma}\label{lem:unknotTRP}
    Every non-loose Legendrian unknot has the TRP.
\end{lemma}

\begin{proof}
    If a Legendrian unknot is non-loose,  it is an unknot either in a tight contact manifold or in an overtwisted $S^3$.  Legendrian unknots in tight manifolds trivially have the TRP.
    The classification of Legendrian unknots in overtwisted $S^3$'s given in Theorem~\ref{thm:unknotclassification} below lead to Legendrian surgery descriptions of these knots shown in Figure~\ref{fig:legunknots} demonstrating the TRP.  In particular, a $(-1)$ Legendrian surgery on the Legendrian unknot will cancel the $(+1)$ surgery on the (purple) parallel push-off.  For Figure~\ref{fig:legunknots}(a), this leaves a single $tb=-1$ unknot with a $(+1)$ surgery that yields the tight contact structure on $S^1 \times S^2$. For the Legendrian unknots of Figure~\ref{fig:legunknots}(b) and (c), this leaves a surgery diagram with no $(+1)$ Legendrian surgeries; hence the resulting manifold is tight.
\end{proof}

The classification of Legendrian unknots up to Legendrian isotopy in $(S^3, \xi_{std})$ is due to Eliashberg and Fraser \cite{EliFr}. Non-loose Legendrian unknots in overtwisted  $S^3$’s were classified up to coarse equivalence by Eliashberg and Fraser \cite{EliFr}; see \cite{Etnyre-knotsinOTctctstr} and \cite{GeiOn} for alternative proofs. We say that two Legendrian knots $L_1$ and $L_2$ are {\em coarsely equivalent} if there is a contactomorphism of the ambient manifold carrying  $L_1$ to $L_2$.

\begin{theorem}[Classification of Legendrian unknots, \cite{EliFr}]
\label{thm:unknotclassification} 
\phantom{xxxxxxxx}\\
\begin{itemize}
  \item[(i)] Let $L$ be an oriented Legendrian unknot in $(S^3, \xi_{std})$. Then the classical invariants $(tb, rot)$ determine $L$ up to Legendrian isotopy. All Legendrian unknots are stabilizations of the unique Legendrian unknot with $\tb= -1$ and $\rot=0$.  For each negative integer $n \leq -1$, we have $|n|$ distinct Legendrian unknots with $\tb=n$. The stabilization of $L$ is obtained by replacing the box in  Figure~\ref{fig:legunknots} with a sequence of $n-1$ stabilizations, each of type $s$ or $z$ shown on the bottom of  Figure~\ref{fig:legunknots}. If the clockwise orientation is chosen for $L$, the type $s$ stabilization gives a negative stabilization, and the type $z$ stabilization gives a positive stabilization of $L$. (The type $s$ and $z$ stabilizations commute, so the order of the sequence of stabilizations does not matter.)
\begin{figure}[h!]
    \centering
    \includegraphics{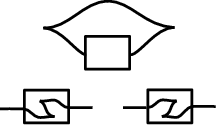}
    \caption{If the clockwise orientation is chosen for $L$, the type $s$ stabilization (left) gives a negative stabilization, and the type $z$ stabilization (right) gives a positive stabilization of $L$.
    }
    \label{fig:legunknots}
\end{figure}  
    \item[(ii)] 
    Let $L$ be an oriented non-loose Legendrian unknot in an overtwisted contact structure $\xi$ on $S^3$. Then $\xi$ is the contact structure $\xi_{1/2}$ and the invariants $(\tb(L), \rot(L)) \in \{(n,\pm(n - 1)) : n \in \mathbb{N}\}$ determine $L$ up to coarse equivalence. 
    These are illustrated in Figure~\ref{fig:nonlooseunknots-surgerydescription}. 
    
 \begin{figure}[h!]
    \centering
    \includegraphics{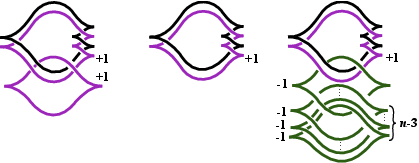}
   \caption{Legendrian surgery descriptions of non-loose Legendrian unknots.  The left knot $L$ has $(\tb, \rot) = (1, 0)$.  Depending on a choice of orientation of $L$, the middle knot $L$ has $(\tb, \rot) = (2, \pm1)$ and the right one has $(\tb, \rot) = (n, \pm(n - 1)), n \geq 3$.}
    \label{fig:nonlooseunknots-surgerydescription}
 \end{figure}
\end{itemize}
\end{theorem}

\begin{cor}[Classification of unoriented Legendrian unknots]
\label{cor:unknotclassificationunoriented} 
\phantom{xxxxxxxx}\\
\begin{itemize}
  \item[(i)] Let $L$ be an unoriented Legendrian unknot in $(S^3, \xi_{std})$. All Legendrian unknots are stabilizations of the unique Legendrian unknot with $\tb= -1$. Up to coarse equivalence, for each positive integer $k \geq 1$, we have $k$ distinct Legendrian unknots with even $\tb=- 2k$, and we have $k$ distinct Legendrian unknots with odd $\tb= 1 - 2k$.
 
    \item[(ii)] Let $L$ be an unoriented non-loose Legendrian unknot in $(S^3, \xi_{1/2})$. Up to coarse equivalence, for each positive integer $k \geq 1$, we have a unique non-loose Legendrian unknot with $\tb = k$.
\end{itemize}
\end{cor}

\subsection{Generalized Hopf Links}

For $n\in\Z$, let $H_n$ be the annulus whose boundary is the anti-parallel $(2, -2n)$-torus link as shown in Figure~\ref{fig:HopfandHopfPlumbing}(Left).  So $H_1$ is the positive Hopf band, $H_{-1}$ is the negative Hopf band, and $H_0$ is the planar annulus. We say $H_n$ is a {\em generalized Hopf band} and its boundary $\bdry H_n$ is a {\em generalized Hopf link}.

\begin{figure}
    \centering
    \includegraphics[width=.3\textwidth]{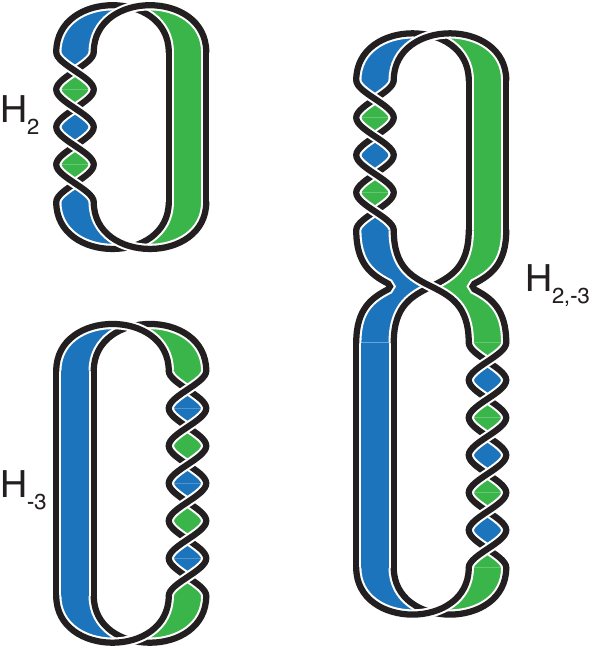}
    \caption{(Left) The generalized Hopf bands $H_2$ and $H_{-3}$.  (Right) The plumbing $H_{2,-3}$ of $H_2$ with $H_{-3}$.}
    \label{fig:HopfandHopfPlumbing}
\end{figure}

\begin{theorem}\label{thm:genhopf}
For $n\neq 0$,  let $\Lambda_n$ be a non-loose Legendrian unknot with $tb=-n$ in $(S^3,\xi)$ 
where $\xi = \xi_{std}$ if $n>0$ and $\xi=\xi_{1/2}$ if $n<0$.
Then $R(\Lambda_n)$ is topologically the generalized Hopf band $H_n$, and the transverse generalized Hopf link $\bdry R(\Lambda_n)$ is non-loose.  

Conversely, suppose for $n\neq 0$ that $\Lambda$ is a Legendrian graph in $(S^3, \xi)$ such that $R(\Lambda)$ is topologically the generalized Hopf band $H_n$ and the transverse generalized Hopf link $\bdry R(\Lambda)$ is non-loose.
Then $\Lambda$ deformation retracts to a non-loose Legendrian unknot with $tb=-n$ and $\xi = \xi_{std}$ if $n>0$ and $\xi=\xi_{1/2}$ if $n<0$.
\end{theorem}

\begin{proof}
Since $\tb(\Lambda_n) = -n$, a positive push-off of $\Lambda_n$ has linking number $-n$ with $\Lambda_n$.  Thus, if they were oriented coherently, the components of $\bdry R(\Lambda_n)$ would be the $(2,-2n)$ torus link.  Hence  $R(\Lambda_n)$ is the generalized Hopf band $H_n$.
The tightness of the complement of the transverse link $\bdry R(\Lambda_n)$ follows from Theorem~\ref{thm:main} since $\Lambda_n$ has the TRP as discussed in Lemma~\ref{lem:unknotTRP}.

For the converse, suppose for some Legendrian graph $\Lambda$ in $(S^3, \xi)$ that $\bdry R(\Lambda)$ is non-loose and $R(\Lambda)$ is isotopic to the generalized Hopf band $H_n$.  As the deformation retract of $\Lambda$ is also a spine for $R(\Lambda)$, Lemma~\ref{lem:LegRealizationGraphs} shows we may assume $\Lambda$ has no leaves.  Thus, as a spine of $H_n$, we may assume $\Lambda$ is a Legendrian unknot with $tb=-n$.  Since $\bdry R(\Lambda)$ is non-loose, Lemma~\ref{lem:loosenessofribbonsandspines} implies that $\Lambda$ is non-loose.  The result now follows from the classification of non-loose Legendrian unknots, Theorem~\ref{thm:unknotclassification}.
\end{proof}

\begin{remark}
    When $n\geq 3$, we have multiple transverse generalized Hopf links representing the same topological generalized Hopf link $\bdry H_n$.  These correspond to having multiple Legendrian isotopy classes of unoriented Legendrian unknots, see Corollary~\ref{cor:unknotclassificationunoriented}.  They can be distinguished by looking at the self-linking of their individual components.
    
    When $n \leq -1$, there is a unique (up to contactomorphism)  unoriented non-loose $\Lambda_n$ and hence unique (up to contactomorphism)  non-loose transverse generalized Hopf link representing $\bdry H_n$.
\end{remark}

\subsection{Double twist knots and further Murasugi sums}
 
Define $H_{m,n}$ to be the  once-punctured torus embedded in $S^3$ that is the plumbing of the generalized Hopf bands $H_m$ and $H_n$ as in Figure~\ref{fig:HopfandHopfPlumbing}(Right).
The  knots $K_{m,n}=\bdry H_{m,n}$ are known as {\em double twist knots} for $m,n\neq 0$.
If either $m=0$ or $n=0$, then  $H_{m,n}$ compresses and $K_{m,n}$ is the unknot.  Otherwise $K_{m,n}$ is a genus $1$ knot and $H_{m,n}$ is an incompressible Seifert surface.  This is the unique incompressible Seifert surface if $m=\pm1$ or $n=\pm1$. When $|m|>1$ and $|n|>1$, there are two non-isotopic ones related by changing the plumbing disk.  (Since the double twist knots are two-bridge knots, this classification of incomporessible Seifert surfaces follows from \cite{HT-2bridge} for example.)

For $n\neq 0$, let $\Lambda_n$ be a non-loose Legendrian unknot with $tb=-n$ in $(S^3, \xi)$.  
Then for $m\neq0$ and $n\neq0$, we define $\Lambda_{m,n}$ to be a Legendrian graph $\Lambda_m * \Lambda_n$ that is a fusion of $\Lambda_m$ and $\Lambda_n$ of order $2$.  Examples of such fusions are depicted via Legendrian surgery description in Figure~\ref{fig:Legendrianplumbingspine}.  Topologically, $H_{m,n} = R(\Lambda_{m,n})$.  

\begin{figure}
    \centering
    \includegraphics[width=\textwidth]{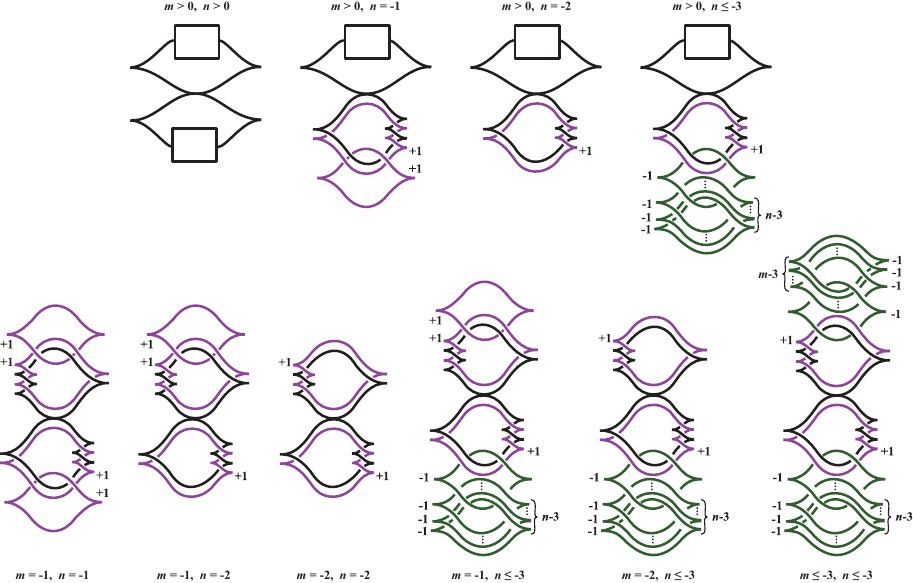}
    \caption{Legendrian surgery descriptions of examples of Legendrian graphs $\Lambda_{m,n}$}
    \label{fig:Legendrianplumbingspine}
\end{figure}

\begin{theorem}\label{thm:doubletwistknots}
A Legendrian graph $\Lambda_{m,n}$ for non-zero integers $m, n$
 is a non-loose Legendrian graph in $(S^3, \xi)$ for 
 \begin{itemize}
     \item $\xi = \xi_{std}$ if $m, n>0$,
    \item $\xi=\xi_{1/2}$ if $m>0>n$ or $n>0>m$, and
    \item $\xi= \xi_{3/2}$ if $0>m, n$.
 \end{itemize}

Furthermore, the transverse double twist knot $K_{m,n} = \bdry R(\Lambda_{m,n})$ is non-loose. 
\end{theorem}

\begin{proof}
By Theorem~\ref{thm:unknotclassification}, $\Lambda_n$ is a Legendrian unknot in $(S^3, \xi)$ where $\xi=\xi_{std})$ if and only if $n>0$ while $\xi=\xi_{1/2}$ if and only if $n<0$.  As $\Lambda_{m,n}$ is in the connected sum of the contact manifolds containing $\Lambda_m$ an $\Lambda_n$, we obtain the classification of contact manifolds containing $\Lambda_{m,n}$ as stated.

Since $\Lambda_m$ and $\Lambda_n$ have the TRP by Lemma~\ref{lem:unknotTRP},
 Theorem~\ref{thm:murasugisums} implies that $\Lambda_{m,n}$ also has the TRP.   Hence Theorem~\ref{thm:main} shows that $\bdry R(\Lambda_{m,n})$ is non-loose.  
\end{proof}

More generally, let $\mathcal{H}$ be the minimal set of Seifert surfaces in $S^3$ that contains the generalized Hopf bands $H_n$ for $n\neq 0$ and is closed under Murasugi sum. Thus, any $H\in \mathcal{H}$ Murasugi de-sums into a collection of generalized Hopf bands.  Let $\nu(H)$ be the number of generalized Hopf bands $H_n$ in this collection with $n<0$.
\begin{theorem}
    If $H \in \mathcal{H}$ with $\nu = \nu(H)$, then there is a Legendrian graph $\Lambda$ with the TRP in $(S^3, \xi)$ 
    where $\xi = \xi_{std}$ if $\nu=0$ and $\xi=\xi_{\nu-1/2}$ if $\nu >0$
    such that $R(\Lambda)$ is topologically isotopic to $H$ and the transverse link $\bdry R(\Lambda)$ is non-loose.
\end{theorem}

\begin{proof}
If $H$ is just a generalized Hopf band, then this follows from Theorem~\ref{thm:genhopf}.  So now by induction, assume that $H$ is a Murasugi sum $H_1 \ast H_2$ of two surfaces $H_1$ and $H_2$ in $\mathcal{H}$ where $\nu(H_i) = \nu_i$.  Then $\nu(H)=\nu=\nu_1+\nu_2$.

Furthermore, by induction, $H_i$ is isotopic to ribbon $R_i$ of a Legendrian graph with the TRP in $(S^3, \xi_i)$ for each $i=1,2$.  Since any spine for $R_i$ may be Legendrian realized by Lemma~\ref{lem:legspinesofribbons}, let $\Lambda_i$ be one (for each $i=1,2$) so that the Murasugi sum $H_1 \ast H_2$ is realized as the ribbon of a Legendrian fusion $\Lambda_1 \ast \Lambda_2$.  Thus $H$ is isotopic to the ribbon $R(\Lambda_1 \ast \Lambda_2)$ in $(S^3, \xi_1 \# \xi_2)$.  As $\Lambda_i$ has the same exterior as $R_i$, it has the TRP too. Then by Theorem~\ref{thm:murasugisums} $\Lambda_1 \ast \Lambda_2$ has the TRP, and by Theorem~\ref{thm:main} $\bdry R(\Lambda_1 \ast \Lambda_2)$ is non-loose.

Finally, observe that $\nu = 0$ if and only if $\nu_1=\nu_2=0$ so 
$\xi_1 \# \xi_2 = \xi_{std}$ if $\nu=0$.  Otherwise, since at least one of $\xi_i$ is overtwisted, $\xi_1 \# \xi_2$ is overtwisted and the calculation $\xi_1 \# \xi_2=\xi_{\nu_1+\nu_2-1/2}=\xi_{\nu-1/2}$ follows.
\end{proof}

\begin{remark}
    Though a ribbon $R$ of a Legendrian graph $\Lambda$ may be topologically expressed as a Murasugi sum, it is not clear that some spine of $R$ may be Legendrian realized as a Legendrian fusion.  Furthermore, even if $\Lambda = \Lambda_1 \ast \Lambda_2$ is non-loose, it may be that $\Lambda_1$ or $\Lambda_2$ is loose.
\end{remark}

\section*{Acknowledgements}

 We thank Lev Tovstopyat-Nelip for helpful discussions, sharing the work in progress \cite{HTN}, and pointing out Remark~\ref{rem:trivialthat}.
SO also thanks both the Department of Mathematics and the Institute of the Mathematical Sciences of the Americas (IMSA) at the University of Miami for their hospitality while this work was developed and written.

KLB was partially supported by the Simons Foundation grant \#523883 and gift \#962034.
SO was partially supported by the Turkish Fulbright Commission, IMSA Visitor Program, and T\"UB\.ITAK 2219.

\bibliographystyle{alpha}
\bibliography{biblio}

\newcommand{\etalchar}[1]{$^{#1}$}
\begin{thebibliography}{BEVHM12}

\bibitem[Bal08]{baldwin-comultiplicity}
John~A. Baldwin.
\newblock Comultiplicativity of the {O}zsv\'{a}th-{S}zab\'{o} contact
  invariant.
\newblock {\em Math. Res. Lett.}, 15(2):273--287, 2008.

\bibitem[Bal12]{baldwin-steincobordisms}
John~A. Baldwin.
\newblock Contact monoids and {S}tein cobordisms.
\newblock {\em Math. Res. Lett.}, 19(1):31--40, 2012.

\bibitem[BCV09]{BCV-LegRibbonsinOTctctstr}
Sebastian Baader, Kai Cieliebak, and Thomas Vogel.
\newblock Legendrian ribbons in overtwisted contact structures.
\newblock {\em J. Knot Theory Ramifications}, 18(4):523--529, 2009.

\bibitem[BEH{\etalchar{+}}15]{BEHKVHM}
Inanc Baykur, John Etnyre, Matthew Hedden, Keiko Kawamuro, and Jeremy Van
  Horn-Morris.
\newblock Official report of the 2nd square meeting.
\newblock In {\em Contact and symplectic geometry and the mapping class groups,
  American Institute of Mathematics}, July 2015.

\bibitem[Ben83]{bennequin}
Daniel Bennequin.
\newblock Entrelacements et \'{e}quations de {P}faff.
\newblock In {\em Third {S}chnepfenried geometry conference, {V}ol. 1
  ({S}chnepfenried, 1982)}, volume 107-108 of {\em Ast\'{e}risque}, pages
  87--161. Soc. Math. France, Paris, 1983.

\bibitem[BEVHM12]{BEVHM}
Kenneth~L. Baker, John~B. Etnyre, and Jeremy Van Horn-Morris.
\newblock Cabling, contact structures and mapping class monoids.
\newblock {\em J. Differential Geom.}, 90(1):1--80, 2012.

\bibitem[BI09]{BaaderIshikawa}
Sebastian Baader and Masaharu Ishikawa.
\newblock Legendrian graphs and quasipositive diagrams.
\newblock {\em Ann. Fac. Sci. Toulouse Math. (6)}, 18(2):285--305, 2009.

\bibitem[BO15]{BO-nonlooseness}
Kenneth~L. Baker and Sinem Onaran.
\newblock Nonlooseness of nonloose knots.
\newblock {\em Algebr. Geom. Topol.}, 15(2):1031--1066, 2015.

\bibitem[BVVV13]{BVVV}
John~A. Baldwin, David~Shea Vela-Vick, and Vera V\'{e}rtesi.
\newblock On the equivalence of {L}egendrian and transverse invariants in knot
  {F}loer homology.
\newblock {\em Geom. Topol.}, 17(2):925--974, 2013.

\bibitem[Col99]{colin}
Vincent Colin.
\newblock Recollement de vari\'{e}t\'{e}s de contact tendues.
\newblock {\em Bull. Soc. Math. France}, 127(1):43--69, 1999.

\bibitem[DGS04]{DGS}
Fan Ding, Hansj\"{o}rg Geiges, and Andr\'{a}s~I. Stipsicz.
\newblock Surgery diagrams for contact 3-manifolds.
\newblock {\em Turkish J. Math.}, 28(1):41--74, 2004.

\bibitem[EF09]{EliFr}
Yakov Eliashberg and Maia Fraser.
\newblock Topologically trivial {L}egendrian knots.
\newblock {\em J. Symplectic Geom.}, 7(2):77--127, 2009.

\bibitem[Eli92]{eliashberg}
Yakov Eliashberg.
\newblock Contact {$3$}-manifolds twenty years since {J}. {M}artinet's work.
\newblock {\em Ann. Inst. Fourier (Grenoble)}, 42(1-2):165--192, 1992.

\bibitem[Etn03]{Etnyre-IntroContactLect}
John~B. Etnyre.
\newblock Introductory lectures on contact geometry.
\newblock In {\em Topology and geometry of manifolds ({A}thens, {GA}, 2001)},
  volume~71 of {\em Proc. Sympos. Pure Math.}, pages 81--107. Amer. Math. Soc.,
  Providence, RI, 2003.

\bibitem[Etn06]{Etnyre-LecturesOpenBooks}
John~B. Etnyre.
\newblock Lectures on open book decompositions and contact structures.
\newblock In {\em Floer homology, gauge theory, and low-dimensional topology},
  volume~5 of {\em Clay Math. Proc.}, pages 103--141. Amer. Math. Soc.,
  Providence, RI, 2006.

\bibitem[Etn13]{Etnyre-knotsinOTctctstr}
John~B. Etnyre.
\newblock On knots in overtwisted contact structures.
\newblock {\em Quantum Topol.}, 4(3):229--264, 2013.

\bibitem[EVV10]{EtnyreVelaVick}
John~B. Etnyre and David~Shea Vela-Vick.
\newblock Torsion and open book decompositions.
\newblock {\em Int. Math. Res. Not. IMRN}, (22):4385--4398, 2010.

\bibitem[Gab83a]{gabai}
David Gabai.
\newblock Foliations and the topology of {$3$}-manifolds.
\newblock {\em J. Differential Geom.}, 18(3):445--503, 1983.

\bibitem[Gab83b]{gabai-murasugi}
David Gabai.
\newblock The {M}urasugi sum is a natural geometric operation.
\newblock In {\em Low-dimensional topology ({S}an {F}rancisco, {C}alif.,
  1981)}, volume~20 of {\em Contemp. Math.}, pages 131--143. Amer. Math. Soc.,
  Providence, RI, 1983.

\bibitem[Gei08]{GeiBook}
Hansj\"{o}rg Geiges.
\newblock {\em An introduction to contact topology}, volume 109 of {\em
  Cambridge Studies in Advanced Mathematics}.
\newblock Cambridge University Press, Cambridge, 2008.

\bibitem[Gir91]{girouxconvex}
Emmanuel Giroux.
\newblock Convexit\'{e} en topologie de contact.
\newblock {\em Comment. Math. Helv.}, 66(4):637--677, 1991.

\bibitem[GO15]{GeiOn}
Hansj\"{o}rg Geiges and Sinem Onaran.
\newblock Legendrian rational unknots in lens spaces.
\newblock {\em J. Symplectic Geom.}, 13(1):17--50, 2015.

\bibitem[Hay21]{Hayden-QPlinksandSteinsfces}
Kyle Hayden.
\newblock Quasipositive links and {S}tein surfaces.
\newblock {\em Geom. Topol.}, 25(3):1441--1477, 2021.

\bibitem[Hay22]{Hayden-LegribbonsandSQP}
Kyle Hayden.
\newblock Legendrian ribbons and strongly quasipositive links in an open book.
\newblock {\em J. Math. Pures Appl. (9)}, 165:42--57, 2022.

\bibitem[HKM02]{HKM-convexdecomptheory}
Ko~Honda, William~H. Kazez, and Gordana Mati\'{c}.
\newblock Convex decomposition theory.
\newblock {\em Int. Math. Res. Not.}, (2):55--88, 2002.

\bibitem[HKM03]{HKM-tightctctstrfiberedhyp3mfld}
Ko~Honda, William~H. Kazez, and Gordana Mati\'{c}.
\newblock Tight contact structures on fibered hyperbolic 3-manifolds.
\newblock {\em J. Differential Geom.}, 64(2):305--358, 2003.

\bibitem[HKM09]{HKM-ctctinvtSFH}
Ko~Honda, William~H. Kazez, and Gordana Mati\'{c}.
\newblock The contact invariant in sutured {F}loer homology.
\newblock {\em Invent. Math.}, 176(3):637--676, 2009.

\bibitem[Hon00]{honda-classificationI}
Ko~Honda.
\newblock On the classification of tight contact structures. {I}.
\newblock {\em Geom. Topol.}, 4:309--368, 2000.

\bibitem[HT22]{HT-2bridge}
A.~Hatcher and W.~Thurston.
\newblock Incompressible surfaces in 2-bridge knot complements.
\newblock In {\em Collected works of {W}illiam {P}. {T}hurston with commentary.
  {V}ol. {II}. 3-manifolds, complexity and geometric group theory}, pages
  171--192. Amer. Math. Soc., Providence, RI, [2022] \copyright 2022.
\newblock Reprint of [0778125].

\bibitem[HTN23]{HTN}
Matt Hedden and Lev Tovstopyat-Nelip.
\newblock Quasipositive surfaces and {F}loer homology.
\newblock In preparation, 2023.

\bibitem[IK19]{ItoKawamuro}
Tetsuya Ito and Keiko Kawamuro.
\newblock The defect of the {B}ennequin-{E}liashberg inequality and {B}ennequin
  surfaces.
\newblock {\em Indiana Univ. Math. J.}, 68(3):799--833, 2019.

\bibitem[LM18]{licatamathews}
Joan~E. Licata and Daniel~V. Mathews.
\newblock Morse structures on partial open books with extendable monodromy.
\newblock In {\em 2016 {MATRIX} annals}, volume~1 of {\em MATRIX Book Ser.},
  pages 287--303. Springer, Cham, 2018.

\bibitem[LOSS09]{LOSS}
Paolo Lisca, Peter Ozsv\'{a}th, Andr\'{a}s~I. Stipsicz, and Zolt\'{a}n
  Szab\'{o}.
\newblock Heegaard {F}loer invariants of {L}egendrian knots in contact
  three-manifolds.
\newblock {\em J. Eur. Math. Soc. (JEMS)}, 11(6):1307--1363, 2009.

\bibitem[Mas12]{massot}
Patrick Massot.
\newblock Infinitely many universally tight torsion free contact structures
  with vanishing {O}zsv\'{a}th-{S}zab\'{o} contact invariants.
\newblock {\em Math. Ann.}, 353(4):1351--1376, 2012.

\bibitem[OS05]{OzSz-HFctct}
Peter Ozsv\'{a}th and Zolt\'{a}n Szab\'{o}.
\newblock Heegaard {F}loer homology and contact structures.
\newblock {\em Duke Math. J.}, 129(1):39--61, 2005.

\bibitem[Rud05]{rudolph}
Lee Rudolph.
\newblock Knot theory of complex plane curves.
\newblock In {\em Handbook of knot theory}, pages 349--427. Elsevier B. V.,
  Amsterdam, 2005.

\bibitem[SV09]{stipsicz-vertesi}
Andr\'{a}s~I. Stipsicz and Vera V\'{e}rtesi.
\newblock On invariants for {L}egendrian knots.
\newblock {\em Pacific J. Math.}, 239(1):157--177, 2009.

\bibitem[TN20]{TN2020Transverse}
Lev Tovstopyat-Nelip.
\newblock On the transverse invariant and braid dynamics, 2020.

\bibitem[TN22]{LevQPDecompLagrang}
Lev Tovstopyat-Nelip.
\newblock Quasipositive surfaces and decomposable {L}agrangians.
\newblock November 2022.

\bibitem[Wan15]{wand}
Andy Wand.
\newblock Tightness is preserved by {L}egendrian surgery.
\newblock {\em Ann. of Math. (2)}, 182(2):723--738, 2015.

\end{thebibliography}
\end{document}